\documentclass[article, 11pt]{amsart}

\usepackage{amsmath, amssymb, amsfonts, amsthm, latexsym}
\usepackage{tikz}

\newtheorem{Theorem}{Theorem}[section]
\newtheorem{Lemma}[Theorem]{Lemma}

\newtheorem{Proposition}[Theorem]{Proposition}

\newtheorem{Question}[Theorem]{Question}

\theoremstyle{definition}
\newtheorem{Remark}[Theorem]{Remark}

\newtheorem{Definition}[Theorem]{Definition}
\newtheorem{Notation}[Theorem]{Notation}
\newtheorem{acknowledgement}{Acknowledgement}

\def\dim{\operatorname{dim}}

\def\hdeg{\operatorname{hdeg}}
\def\Ext{\operatorname{Ext}}

\def\lim{\operatorname{lim}}

\def\mm{{\frak m}}

\begin{document}
	
	\title[Generalized Cohen-Macaulay rings]{Small perturbations in generalized Cohen-Macaulay local rings}
	
	\author[Pham Hung Quy]{Pham Hung Quy}
	\address{Department of Mathematics, FPT University, Hanoi, Vietnam}
	\email{quyph@fe.edu.vn}
	
	\author{Van Duc Trung}
	\address{Department of Mathematics, University of Genoa, Via Dodecaneso 35, 16146 Genoa, Italy}
	\email{vanductrung@dima.unige.it}
	
	\thanks{
		The work is partially supported by a fund of Vietnam National Foundation for Science
		and Technology Development (NAFOSTED) under grant number 101.04-2020.10} 
	
	\begin{abstract} Let $(R, \frak m)$ be a generalized Cohen-Macaulay local ring of dimension $d$, and $f_1, \ldots, f_r$ a part of system of parameters of $R$. In this paper we give explicit numbers $N$ such that the lengths of all lower local cohomology modules and the Hilbert function of $R/(f_1, \ldots, f_r)$ are preserved when we perturb the sequence $f_1, \ldots, f_r$ by $\varepsilon_1, \ldots, \varepsilon_r \in \frak m^N$. The second assertion extends a previous result of Srinivas and Trivedi for generalized Cohen-Macaulay rings.
	\end{abstract}
	
	\maketitle
	
	\section{Introduction}
	This work is inspired by the recent work of the first author with Ma and Smirnov \cite{MQS} about the preservation of Hilbert function under sufficiently small perturbations which was also inspired by the previous work of Srinivas and Trivedi \cite{ST1}. Taking a small perturbation arises naturally in studying deformations when we change the defining equations by adding terms of high order. In this way we can transform a singularity defined analytically, e.g., as a quotient of a (convergent) power series ring, into an algebraic singularity by truncating the defining equations.
	
	This problem was first considered  by Samuel in 1956. Let $f \in S = k[[x_1, \ldots, x_d]]$ be a hypersurface with an isolated singularity, i.e. the Jacobian ideal $J(f) = (\frac{\partial f}{\partial x_1}, \ldots, \frac{\partial f}{\partial x_d})$ is $(x_1,\ldots, x_d)$-primary. Then Samuel proved that for every $\varepsilon \in (x_1,\ldots, x_d) J(f)^2$ we have an automorphism of $S$ that maps $f \mapsto f + \varepsilon$. In particular, Samuel’s result asserts if $f$ has an isolated singularity and $\varepsilon$ is in a sufficiently large power of $(x_1, \ldots, x_d)$, then the rings $S/(f)$ and $S/(f + \varepsilon)$ are isomorphic. Samuel's result was extended by Hironaka in 1965, who showed that if $S/I$ is an equidimensional reduced isolated singularity, then $S/I \cong S/I'$ for every ideal $I'$ obtained by changing the generators of $I$ by elements of sufficiently large order such that $S/I'$ is still reduced, equidimensional, and same height as $I$.

	The isolated singularity is essential in the both theorems of Samuel and Hironaka. For a local ring $(R, \frak m)$ and a sequence of elements $f_1, \ldots, f_r$, instead of requiring the deformation to give isomorphic rings $R/(f_1, \ldots, f_r) \cong R/(f_1 + \varepsilon_1, \ldots, f_r + \varepsilon_r)$, we consider a weaker question: what properties and invariants are preserved by a sufficiently fine perturbation? For example, Eisenbud \cite{E} showed how to control the homology of a complex under a perturbation and thus showed that Euler characteristic and depth can be preserved. As an application, if $f_1, \ldots, f_r$ is a regular sequence, then so is the sequence $f_1 + \varepsilon_1, \ldots, f_r + \varepsilon_r$ as long as we take a sufficiently small perturbation. Huneke and Trivedi \cite{HT} extended this result for filter regular sequences, a generalization of the notion of regular sequence.
	
	For numerical invariants, perhaps the most natural direction is to study the behavior of Hilbert function. Srinivas and Trivedi \cite{ST1} showed that the Hilbert function of a sufficiently fine perturbation is at most the original Hilbert function. Furthermore they proved that the Hilbert functions of $R/(f_1, \ldots, f_r)$ and $R/(f_1 + \varepsilon_1, \ldots, f_r + \varepsilon_r)$ coincide under small perturbations provided two conditions: (a) $f_1, \ldots, f_r$ a filter regular sequence; (b) $R/(f_1, \ldots, f_r)$ is generalized Cohen-Macaulay. Recalling that $(R, \frak m)$ is generalized Cohen-Macaulay if all lower local cohomology $H^i_{\frak m}(R), i < \dim R$, have finite length. Moreover, a generalized Cohen-Macaulay ring is Cohen-Macaulay on the punctured spectrum. Srinivas and Trivedi gave examples to show that the condition (a) is essential even if $f_1, \ldots, f_r$ is a part of system of parameters. However they asked whether the condition (b) is superfluous.
	
	\begin{Notation} Let $(R,\mm)$ be a Noetherian local ring and $I = (f_1,\ldots,f_r)$ an ideal of $R$. For each $N > 0$ we denote
		$$C_N(I) = \{ (f_1+\varepsilon_1, \ldots, f_r+\varepsilon_r)\ | \ \varepsilon_i \in \mm^N \}.$$
	\end{Notation}
	
	Recently, Ma, Smirnov and the first author \cite{MQS} answered affirmatively the above question of Srinivas and Trivedi and proved the following.
	\begin{Theorem}
		Let $(R, \frak m)$ be a Noetherian local ring of dimension $d$, and $I \subseteq R$ is generated by  a filter regular sequence $f_1,\ldots,f_r$. Then there exists $N>0$ such that for all $I_N \in C_N(I)$, the Hilbert functions of
		$R/I$ and $R/I_N$ are equal, i.e.
		$$\ell(R/(I + \mm^n)) = \ell(R/(I_N + \frak m^n))$$
		for all $n \ge 1$.\footnote{Actually, we proved the result for any ideal $J$ such that $(f_1, \ldots, f_r)+J$ is $\frak m$-primary. Although the main result of this paper can be extended for such ideals, we will keep our interest for the maximal ideal for simplicity.}
	\end{Theorem}
	We also asked the question.
	\begin{Question}\label{question}
		Can one obtain explicit bounds on $N$?
	\end{Question}
	A certainly positive answer for the case $r=1$ was given in \cite[Theorem 3.3]{MQS}. If $R$ is a Cohen-Macaulay local ring of dimension $d$, Srinivas and Trivedi \cite[Proposition 1.1]{ST2} provided a formula for $N$ in terms of the multiplicity for any $r \ge 1$. Namely, we can choose
	$$N = (d-r)! \, e(R/I) + 2.$$
	Inspired by the above formula, one can hope to give a bound for $N$ in any local ring by using the extended degree instead of the multiplicity. See the next section for more details about the notion of extended degree. The aim on the present paper is to give an evident for this belief. We will extend the above result of Srinivas and Trivedi for the class of generalized Cohen-Macaulay rings by using the multiplicity and the length of local cohomology $H^i_{\frak m}(R)$.
	
	Let $(R, \frak m)$ be a local ring and $M$ a generalized Cohen-Macaulay module of dimension $d$. The Buchsbaum invariant of $M$ is defined as follows
	$$I(M) = \sum_{i = 0}^{d-1} \binom{d-1}{i} \ell(H^i_{\frak m}(M)).$$
	
	We now present the first main result of this paper.
	\begin{Theorem}\label{result 1}
		Let $(R, \mm)$ be a generalized Cohen-Macaulay local ring of dimension $d$ and $I \subseteq R$ is generated by a part of system of parameters $f_1,\ldots,f_r$ of $R$. Let $s = d-r $, and
		$$N = s!\, \big(e(R/I) + I(R/I)\big) + (s+1)I(R) + 1.$$
		Then for all $I_N \in C_N(I)$ we have
		the Hilbert functions of
		$R/I$ and $R/I_N$ are equal.
	\end{Theorem}
	The method of our proof of the above result is inspired by the Srinivas and Trivedi one in the Cohen-Macaulay case. Let us mention the most important step in our proof. If $R$ is Cohen-Macaulay and $J = (x_1, \ldots, x_s)$ a minimal reduction of $\frak m$ with respect to $R/I$, then we can choose $N$ such that $I + J = I_N + J$ for all $I_N \in C_N(I)$. The strategy of Srinivas and Trivedi was to transform the Hilbert functions of $R/I$ and $R/I_N$ (with respect to $\frak m$) to the Hilbert functions of $R/I$ and $R/I_N$ with respect to the parameter ideal $J$, and using the following well-known fact for Cohen-Macaulay rings
	$$\ell (R/(I+J^{n+1})) = \binom{n+s}{s} \ell(R/(I+J)) = \binom{n+s}{s} \ell(R/(I_N+J)) = \ell (R/(I_N+J^{n+1})).$$
	For generalized Cohen-Macaulay rings, we also have an explicit formula for the Hilbert function with respect to special parameter ideals, say standard parameter ideals, in terms of the length of lower local cohomology modules (see Theorem \ref{HF of standard}). Therefore we need to control $\ell (H^i_{\frak m}(R/I))$ under sufficiently small perturbations. This is the second main result of this paper.
	\begin{Theorem}\label{result 2}
		Let $(R,\mm)$ be a generalized Cohen-Macaulay ring of dimension $d$ and $I \subseteq R$ is generated by a part of system of parameters $f_1,\ldots,f_r$. Let $N = e(R/I) + I(R) + 1$, then for all $I_N \in C_N(I)$ we have
		$$\ell(H^i_{\mm}(R/I)) = \ell(H^i_{\mm}(R/I_N))$$
		for every $i < d-r$.
	\end{Theorem}
	The paper is organized as follows: In the next section, we recall some notations used in this paper. We will prove Theorem \ref{result 2} in Section 3. Section 4 is devoted to prove Theorem \ref{result 1}.
	
	\begin{acknowledgement} This project resulted from a trip the first author took to the University of Genoa,
		we would like to thank Matteo Varbaro for making that trip possible. This paper was written while the first author visited the Vietnam Institute for Advanced Study in Mathematics (VIASM), he would like to thank the VIASM for the very kind support and hospitality. The authors are grateful to Professor Maria Evelina Rossi for her useful discussions on the first step of the project.
	\end{acknowledgement}
	
	\section{Preliminaries}
	
	Throughout this paper, $(R, \frak m)$ denotes a Noetherian local ring. Let $M$ be a finitely generated $R$-module of dimension $d \geq 1$, and $J$ an $\mm$-primary ideal of $R$. The {\it Hilbert function} of $M$ with respect to $J$ is defined by
	$$HF_J(M)(n) := \ell(M/J^{n+1}M)$$
	for all $n\geq 0$.
	
	The Hilbert function of $M$, denoted by $HF(M)$, is the Hilbert function of $M$ with respect to the maximal ideal $\mm$. It is well-known that for $n$ sufficiently large the Hilbert function $HF_J(M)$ becomes a polynomial in $n$ of degree $d$, and can be written as the following form
	$$HF_J(M)(n) = e_0(J,M)\binom{n+d}{d} - e_1(J,M)\binom{n+d-1}{d-1} + \cdots + (-1)^de_d(J,M)$$
	for all $n \gg 0$, where $e_i(J,M)$ are the integers and they are called the Hilbert coefficients of $M$ with respect to $J$. In particular, $e(J,M) = e_0(J, M)$ is called the {\it multiplicity} of $M$ with respect to $J$ and $e(M) = e(\mm,M)$ is called the multiplicity of $M$.
	
	If $M$ is Cohen-Macaulay and $J$ is a parameter ideal we have
	$$HF_J(M)(n) = \ell (M/JM) \, \binom{n+d}{d}$$
	for all $n \ge 0$. In particular $e(J,M) = \ell (M/JM)$. In general we always have the inequality $e(J,M) \le \ell (M/JM)$ for all parameter ideals $J$ of $M$.
	\begin{Definition}
		An $R$-module $M$ is called {\it generalized Cohen-Macaulay} if the difference $\ell (M/JM) - e(J, M)$ is bounded above for every parameter ideal $J$.
	\end{Definition}
	We next recall some well-known facts in the theory of generalized Cohen-Macaulay modules (see \cite{T}).
	\begin{Remark} Let $M$ be an $R$-module of dimension $d$. Then
		\begin{enumerate}
			\item $M$ is generalized Cohen-Macaulay if and only if $H^i_{\frak m}(M)$ has finite length for every $i<d$. Moreover, we have
			$$\ell (M/JM) - e(J, M) \le \sum_{i=0}^{d-1}\binom{d-1}{i} \ell(H^i_{\frak m}(M))$$
			for all parameter ideals $J$. The left hand side, denoted by $I(M)$, and is called the {\it Buchsbaum invariant} of $M$.
			\item If $M$ is generalized Cohen-Macaulay, then for every part of system of parameters $x_1, \ldots, x_r$ we have $I(M) \ge I(M/(x_1, \ldots, x_r)M)$.
			\item If $M$ is generalized Cohen-Macaulay, then every system of parameter is a filter regular sequence of $M$. Recalling that $x_1, \ldots, x_t \in \frak m$ is called a {\it filter regular sequence} of $M$ if
			$$\mathrm{Supp} \big(\frac{(x_1, \ldots, x_{i-1})M: x_i}{(x_1, \ldots, x_{i-1})M} \big) \subseteq \{\frak m\}$$
			for all $i = 1, \ldots, t$.
		\end{enumerate}
	\end{Remark}
	\begin{Definition} Let $M$ be a generalized Cohen-Macalay module of dimension $d$. A parameter ideal $J$ of $M$ is called {\it standard} if
		$$\ell (M/JM) - e(J, M) = \sum_{i=0}^{d-1} \binom{d-1}{i} \ell(H^i_{\frak m}(M)).$$
	\end{Definition}
	\begin{Remark} Let $M$ be a generalized Cohen-Macalay module of dimension $d$. Then there exists a positive integer $N$ such that $J$ is standard for every parameter ideal $J \subseteq \frak m^N$. In fact we can choose $N = I(M)$.
	\end{Remark}
	The Hilbert function of a standard parameter ideal $J$ can be expressed explicitly as follows (see \cite[Corollary 4.2]{T}).
	\begin{Theorem} \label{HF of standard} $J$ is a standard parameter ideal of $M$ if and only if
		$$HF_J(M)(n) =  \binom{n+d}{d}e(J,M) + \sum_{i=1}^d \sum_{j=0}^{d-i}\binom{n+d-i}{d-i}\binom{d-i-1}{j-1}\ell(H_{\mm}^j(M)),$$
		for all $n \geq 0$.
	\end{Theorem}
	
	In order to capture the complexity of non (generalized) Cohen-Macaulay modules, Vasconcelos et al. \cite{V1, V2} introduced the
	notion of extended degree which is a generalization of the notion of multiplicity. Let $\mathcal{M}(R)$ be the category of finitely generated $R$-modules. An {\it extended degree} on $\mathcal{M}(R)$ is a numerical function $D(\bullet)$ on $\mathcal{M}(R)$ such that the following properties hold for every $R$-module $M \in \mathcal{M}(R)$:
	\begin{enumerate}
		\item $D(M) = D(M/L) + \ell(L)$, where $L = H^0_{\frak m}(M)$,
		
		\item $D(M) \geq D(M/xM)$ for a generic element $x$ of $\mm$,
		
		\item $D(M) = e(M)$ if $M$ is a Cohen-Macaulay module.
	\end{enumerate}
	The prototype of an extended degree is the homological degree defined by Vasconcelos in \cite{V1}. If $R$ is a homomorphic image of a Gorenstein ring $S$ with $\dim S = n$ then the {\it homological degree} of $R$-module $M$ is defined by
	$$\hdeg(M):= e(M) + \sum_{i=0}^{d-1}\binom{d-1}{i}\hdeg(\Ext_S^{n-i}(M,S)).$$
	Recently, Cuong and the first author \cite{CQ} introduced a new extended degree, say the {\it unmixed degree}, and denoted by $\mathrm{udeg}(M)$. The readers are encouraged to \cite{CQ} for more details about the construction. If $M$ is generalized Cohen-Macaulay we have 
	$$\hdeg (M) = \mathrm{udeg} (M) = e(M) + I(M).$$ 
	We close this section with some lemmas that will be useful for the proof of the main results.
	
	\begin{Lemma}\label{presever sop} Let $(R, \frak m)$ be a local ring of dimension $d$ and $I \subseteq R$ is generated by  a part of system of parameters $f_1, \ldots, f_r$. Let $D(\bullet)$ be an any extended degree and set $N = D(R/I) + 1$. Then for every $\varepsilon_1, \ldots, \varepsilon_r \in \frak m^{N}$ we have $f_1 + \varepsilon_1, \ldots, f_r + \varepsilon_r$ is a part of system of parameters of $R$.
	\end{Lemma}

	\begin{proof} Let $x_1, \ldots, x_{d-r}$ be a general sequence of elements of $R/I$. Then
		$$\ell (R/(I + (x_1, \ldots, x_{d-r})) \le D(R/I) = N-1.$$
		This implies that $\frak m^{N-1} \subseteq I + (x_1, \ldots, x_{d-r})$. Therefore for every $\varepsilon_1, \ldots, \varepsilon_r \in \frak m^{N}$ we have
		$$(f_1, \ldots, f_r,x_1, \ldots, x_{d-r}) = (f_1 + \varepsilon_1, \ldots, f_r + \varepsilon_r, x_1, \ldots, x_{d-r}).$$
		Hence $f_1 + \varepsilon_1, \ldots, f_r + \varepsilon_r$ is a part of system of parameters of $R$.
	\end{proof}
	It would be nice if we obtain similar results for (filter) regular sequences instead of system of parameters. These results, if have, will play an important role for an answer for Question \ref{question} in the general case. In the main context of this paper $R$ is generalized Cohen-Macaulay, so these three notions coincide. We will need the following regular \cite[Corollary 3.6]{V1}.
	\begin{Lemma}\label{reduction}
		Let $(R, \frak m)$ be a local ring of dimension $d$ with the infinite residue field. Let $\hdeg(\bullet)$ be the homological degree. Then there exists a minimal reduction $J$ of $\frak m$ with reduction number $\mathrm{r}_J(\frak m) \le (d-r)!\, \hdeg(R) - 1$.
	\end{Lemma}
	\begin{Lemma} \label{prevever reduction} Let $(R, \frak m)$ be a local ring of dimension $d$ and $I \subseteq R$ is generated by  a part of system of parameters $f_1, \ldots, f_r$. Let $J$ be a minimal reduction of $\frak m$ in $R/I$, and $k$ a non-negative integer such that $\mathrm{r}_J(\frak m , R/I) \le k$. Then for all $I_{k+2} \in C_{k+2}(I)$, $J$ is a minimal reduction of $\frak m$ in $R/I_{k+2}$ and $\mathrm{r}_J(\frak m, R/I_{k+2}) \le k+1$.
	\end{Lemma}
	\begin{proof}
		By the assumption we have $\frak m^{k+1} + I =  J\frak m^k + I$. Therefore $\frak m^{k+1} \subseteq J + I$. Hence for every $I_{k+2} \in C_{k+2}(I)$  we have $J + I = J + I_{k+2}$. We are going to prove that $\frak m^{k+2} + I_{k+2} =  J\frak m^{k+1} + I_{k+2}$. We have
		\begin{align*}
			J\frak m^{k+1} + I_{k+2} + \frak m (\frak m^{k+2} + I_{k+2}) &= J\frak m^{k+1} + I_{k+2} + \frak m (\frak m^{k+2} + I)\\
			&= J\frak m^{k+1} + I_{k+2} + \frak m (J\frak m^{k+1} + I)\\
			&= J\frak m^{k+1} + I_{k+2} + \frak m I\\
			&= \frak m( J\frak m^{k}  + I)+ I_{k+2}\\
			&= \frak m( \frak m^{k+1}  + I)+ I_{k+2}\\
			&= \frak m^{k+2}  + I_{k+2} + \frak m I\\
			&= \frak m^{k+2}  + I_{k+2}.
		\end{align*}
		The last equality follows from the fact that $\frak m I \subseteq \frak m^{k+2}  + I = \frak m^{k+2}  + I_{k+2}$. By NAK we have $\frak m^{k+2} + I_{k+2} =  J\frak m^{k+1} + I_{k+2}$. The proof is complete.
	\end{proof}
	
	\section{Local cohomology under small perturbations}
	Let $(R,\mm)$ be a generalized Cohen-Macaulay ring of dimension $d$ and $I \subseteq R$ is generated by  a part of system of parameters $f_1, \ldots, f_r$. In this section we provide a positive integer $N$ depends on $e(R/I)$ and $I(R)$ such that for all $I_N \in C_N(I)$ we have the lengths of $H^i_{\mm}(R/I)$ and $H^i_{\mm}(R/I_N)$ coincide for every $0 \leq i < d-r$.
	
	The proof of the main result is based on the induction on $r$, where $r$ is the length of the sequence $f_1, \ldots, f_r$. First, for the case $r=1$ we have the following proposition.
	\begin{Proposition} \label{one form} Let $(R,\mm)$ be a generalized Cohen-Macaulay ring of dimension $d$, and $f$ a parameter element of $R$. Then for every $\varepsilon \in \mm^{I(R)}$ such that $f + \varepsilon$ is a parameter element of $R$, we have
		$$\ell(H^i_{\mm}(R/(f))) = \ell(H^i_{\mm}(R/(f+\varepsilon)))$$
		for every $i < d-1$.
	\end{Proposition}
	\begin{proof} Following from the short exact sequence
		$$0 \longrightarrow R/(0:f) \overset{f}{\longrightarrow} R \longrightarrow R/(f) \longrightarrow 0$$
		we obtain the following short exact sequence
		$$0 \longrightarrow \frac{H^i_{\mm}(R)}{fH^i_{\mm}(R)} \longrightarrow H^i_{\mm}(R/(f)) \longrightarrow (0:_{H^{i+1}_{\mm}(R)} f)\longrightarrow  0$$
		for every $i < d-1$. Similarly we have the following short exact sequence
		$$0 \longrightarrow \frac{H^i_{\mm}(R)}{(f+\varepsilon)H^i_{\mm}(R)} \longrightarrow H^i_{\mm}(R/(f+\varepsilon)) \longrightarrow (0:_{H^{i+1}_{\mm}(R)} (f+\varepsilon))\longrightarrow  0$$
		for every $i < d-1$. Since $\varepsilon \in \mm^{I(R)}$ we have $\varepsilon\,  H^i_{\mm}(R) = 0$ for all $i <d$. It follows that $(0:_{H^{i+1}_{\mm}(R)} f )\cong (0:_{H^{i+1}_{\mm}(R)} (f+\varepsilon))$ and $\frac{H^i_{\mm}(R)}{fH^i_{\mm}(R)} \cong \frac{H^i_{\mm}(R)}{(f+\varepsilon)H^i_{\mm}(R)}$ for all $i < d-1$. Hence the above two short exact sequences imply
		$$\ell(H^i_{\mm}(R/(f))) = \ell(H^i_{\mm}(R/(f+\varepsilon)))$$
		for all $i < d-1$.
	\end{proof}
	We now present the main result of this section.
	\begin{Theorem} \label{lc}
		Let $(R,\mm)$ be a generalized Cohen-Macaulay ring of dimension $d$ and $I \subseteq R$ is generated by  a part of system of parameters $f_1, \ldots, f_r$. Let $N = e(R/I) + I(R) + 1$, then for all $I_N \in C_N(I)$ we have
		$$\ell(H^i_{\mm}(R/I)) = \ell(H^i_{\mm}(R/I_N))$$
		for every $i < d-r$.
	\end{Theorem}
	\begin{proof} Without loss of generality we will always assume that the residue field is infinite. We proceed by induction on $r$. For $r=1$, we have 
		$$N \geq e(R/(f_1)) + I(R/(f_1)) + 1 = \hdeg(R/(f_1)) + 1.$$ 
		So, by Lemma \ref{presever sop}, $f_1+\varepsilon_1$ is a parameter element for every $\varepsilon_1 \in \mm^N$. Hence we are done by Proposition \ref{one form}.
		
		For $r > 1$ and $I_N = (f_1+\varepsilon_1,\ldots,f_r+\varepsilon_r)$, where $\varepsilon_1,\ldots,\varepsilon_r \in \mm^N$. Let $R_1 = R/(f_1)$. For simplicity, we will identify $f_i$ with its image in $R_1$. Since $N \geq e(R_1/({f_2},\ldots,{f_r})R_1) + I(R_1) + 1$ and $\varepsilon_2,\ldots,\varepsilon_r \in \mm^N$, by induction we get
		$$\ell(H^i_{\mm}(R/(f_1,f_2,\ldots,f_r))) = \ell(H^i_{\mm}(R/(f_1,f_2+\varepsilon_2,\ldots,f_r+\varepsilon_r)))$$
		for every $i < d-r$.  Since $N \geq \hdeg(R/I)+1$ and $\varepsilon_1,\ldots,\varepsilon_r \in \mm^N$, by Lemma \ref{presever sop} we have $f_1,f_2+\varepsilon_2,\ldots,f_r+\varepsilon_r$ and $f_1+\varepsilon_1,f_2+\varepsilon_2,\ldots,f_r+\varepsilon_r$ are the parts of system of parameters of $R$. Let $R_2 = R/(f_2+\varepsilon_2,\ldots,f_r+\varepsilon_r)$, we have $f_1$ and $f_1+\varepsilon_1$ are parameter elements of $R_2$. Moreover, $N \geq I(R_2)$ and $\varepsilon_1 \in \mm^N$, by Proposition \ref{one form} we get
		$$\ell(H^i_{\mm}(R_2/f_1R_2)) = \ell(H^i_{\mm}(R_2/(f_1+\varepsilon_1)R_2))$$
		for every $i < d-r$. That is
		$$\ell(H^i_{\mm}(R/(f_1,f_2+\varepsilon_2,\ldots,f_r+\varepsilon_r))) = \ell(H^i_{\mm}(R/(f_1+\varepsilon_1,f_2+\varepsilon_2,\ldots,f_r+\varepsilon_r)))$$
		for every $i < d-r$. Hence we obtain the desired assertion. The proof is complete.
	\end{proof}
	It is natural to ask the following question in general case.
	\begin{Question} Let $(R, \frak m)$ be a local ring and $I \subseteq R$ is generated by a filter regular sequence  $f_1, \ldots, f_t$ . Does there exist a positive integer $N$ such that for all $I_N \in C_N(I)$ we have 
		$$\ell (H^0_{\frak m}(R/I)) = \ell (H^0_{\frak m}(R/I_N))?$$
	\end{Question}

	\section{Hilbert funcion under small perturbations}
	In this section, let $(R,\mm)$ be a generalized Cohen-Macaulay ring of dimension $d$, and $I \subseteq R$ is generated  by a part of system of parameters $f_1,\ldots,f_r$. We will find an explicitly positive integer $N$ depends on $\hdeg(R/I)$ and $I(R)$ such that for all $I_N \in C_N(I)$ the Hilbert functions of $R/I$ and $R/I_N$ coincide. The following lemma is a special case of Lemma \ref{reduction} and Lemma \ref{prevever reduction}.
	\begin{Lemma}\label{reduction gCM}
		Let $(R,\mm)$ be a generalized Cohen-Macaulay ring of dimension $d$ with the infinite residue field, and $I \subseteq R$ is generated  by a part of system of parameters $f_1,\ldots,f_r$. Let $s = d-r$, and $k = s!\, \hdeg(R/I)+1$. Then there exists a minimal reduction $J$ of $\mm$ in $R/I$ such that
		\begin{enumerate}
			\item $\mm^{k+m} + I = J^{m+1}\mm^{k-1}+ I$ for all $m \geq 0$.
			\item For every $I_k \in C_k(I)$ one has $\mm^{k+m} + I_k = J^{m+1}\mm^{k-1}+ I_k$ for all $m \geq 0$.
		\end{enumerate}
	\end{Lemma}
	\begin{proof}
		By Lemma \ref{reduction} there exists a minimal reduction $J=(x_1,\ldots,x_{d-r})$ of $\frak m$ in $R/I$ such that $\mathrm{r}_J(\frak m , R/I) \le (d-r)!\, \hdeg( R/I) - 1 = k - 2$. Hence $\mm^k + I = J\mm^{k-1}+I$. By Lemma \ref{prevever reduction} one has $J$ is a minimal reduction of $\mm$ in $R/I_k$ and $\mathrm{r}_J(\frak m,R/I_k) \le k-1$.  Hence $\mm^k + I_k = J\mm^{k-1} + I_k$. Therefore, for all $m \ge 0$ we have $\frak m^{k+m} R/I = J^{m+1} \frak m^{k-1} R/I$ and $\frak m^{k+m} R/I_k = J^{m+1} \frak m^{k-1} R/I_k$. The claims are now clear.
	\end{proof}
	
	The following theorem is the main result of this section. It extends the result of Srinivas and Trivedi \cite[Proposition 1]{ST2} for generalized Cohen-Macaulay rings.
	\begin{Theorem}
		Let $(R,\mm)$ be a generalized Cohen-Macaulay ring of dimension $d$, and $I \subseteq R$ is generated  by a part of system of parameters $f_1,\ldots,f_r$. Let $s = d-r $, and $$N = s!\, \hdeg(R/I) + (s+1)I(R) + 1.$$ Then for all $I_N \in C_N(I)$ we have
		$$HF(R/I) = HF(R/I_N).$$
	\end{Theorem}
	\begin{proof} Without loss of generality we may assume that the residue field is infinite. 
		Let $k = s!\, \hdeg(R/I) + 1$, by Lemma \ref{reduction gCM}  there exists ideal $J = (x_1,\ldots,x_s) \subseteq \mm$ such that
		$$\mm^{k+m} + I = J^{m+1}\mm^{k-1}+I$$
		for all $m \geq 0$. Moreover, since $I_N \in C_N(I) \subseteq C_k(I)$ we also have
		$$\mm^{k+m} + I_N = J^{m+1}\mm^{k-1}+I_N$$
		for all $m \geq 0$. Let $t= \max\{I(R),1\}$. For all $0 \le i \le t-1$ set $N_i= k + s(t-1) + i$. We have $N_i \le N$ for all $i \le t-1$. Since $I_N \in C_N(I) \subseteq C_{N_0}(I)$ we have
		$$\mm^i + I = \mm^i + I_N$$
		for all  $ i \leq N_0$. Therefore, it is enough to prove that
		$$\ell \left(\frac{R}{\mm^n+I}\right) = \ell \left(\frac{R}{\mm^n+I_N}\right) \quad \quad (1)$$
		for all $n \geq N_0.$ We will prove it in the following equivalent form
		$$\ell \left(\frac{R}{\mm^{N_i + mt}+I}\right) = \ell \left(\frac{R}{\mm^{N_i + mt}+I_N}\right) \quad \quad (2)$$
		for all $0 \le i \le t-1$ and all $m \ge 0$.
		Set $J' = (x_1^t,\ldots,x_s^t)$, one has $J' \subseteq \frak m^t$ is a standard parameter ideal of $R/I$ and $J' J^{(s-1)(t-1)} = J^{s(t-1)+1}$.
		\vskip 0.2cm 
		\noindent {\bf Claim 1.} For all $0 \le i \le t-1$ and all $m \ge 0$ we have
		$$\mm^{N_i+mt} + I = J'^{m+1}\mm^{N_i-t}+I, $$
		and
		$$\mm^{N_i+mt} + I_N = J'^{m+1}\mm^{N_i-t} + I_N$$
		for all $I_N \in C_N(I)$
		\begin{proof}[Proof of Claim 1] We have
			\begin{eqnarray*}
				\frak m^{N_i + mt}R/I &=& J^{s(t-1)+mt+i+1}\frak m^{k-1}R/I\\
				&=& J'^{m+1} J^{(s-1)(t-1) + i }\frak m^{k-1}R/I\\
				& =& J'^{m+1} \frak m^{N_i-t}R/I
			\end{eqnarray*}
			for all $m \ge 0$. Therefore 
			$$\mm^{N_i+mt} + I = J'^{m+1}\mm^{N_i-t}+I$$
			for all $0 \le i \le t-1$ and all $m \ge 0$. The second assertion can be proved similarly. The Claim is proved.
		\end{proof}
		By Claim 1, in order to prove the equality (2) it is enough to show
		$$\ell \left(\frac{R}{J'^{m+1}\mm^{N_i-t}+I}\right) = \ell \left(\frac{R}{J'^{m+1}\mm^{N_i-t}+I_N}\right) \quad \quad (3)$$
		for all $0 \le i \le t-1$ and all $m \ge 0$. On the other hand, since $N \ge e(R/I) + I(R) + 1$ we have
		$$\ell(H^i_{\mm}(R/I)) = \ell(H^i_{\mm}(R/I_N))$$
		for all $i<s$ by Theorem \ref{lc}. We also have $I + J' = I_N + J'$ since $I + J' \supseteq \mm^{N-1}$. Hence $J'$ is a standard parameter ideal of $R/I_N$ and $e(J',R/I_N)= e(J',R/I)$. By Theorem \ref{HF of standard} we have
		$$\ell \left(\frac{R}{J'^{m+1}+I}\right) = \ell \left(\frac{R}{J'^{m+1} +I_N}\right) \quad \quad (4)$$
		for all $m \ge 0$. Therefore in order to prove the equality (3) it is sufficient to prove that
		$$\ell \left(\frac{J'^{m+1}+I}{J'^{m+1}\mm^{N_i-t}+I}\right) = \ell \left(\frac{J'^{m+1} +I_N}{J'^{m+1}\mm^{N_i-t}+I_N}\right) \quad \quad (5)$$
		for all $0 \le i \le t-1$ and all $m \ge 0$. Let $K' = I + J' = I_N + J'$. We will prove (5) in the following equivalent form
		$$\ell \left(\frac{K'^{m+1}+I}{K'^{m+1}\mm^{N_i-t}+I}\right) = \ell \left(\frac{K'^{m+1} +I_N}{K'^{m+1}\mm^{N_i-t}+I_N}\right) \quad \quad (6)$$
		for all $0 \le i \le t-1$ and all $m \ge 0$.
		\vskip 0.2cm
		\noindent {\bf Claim 2.} For all $m \ge 0$ we have $K'^{m+1} \cap I = IK'^m$ and $K'^{m+1} \cap I_N = I_NK'^m$.
		\begin{proof}[Proof of Claim 2]
			Notice that $x_1^t,\ldots,x_s^t$ forms a d-sequence of $R/I$, by \cite[Theorem 2.1]{H} we have
			$$J'^{m+1} \cap I \subseteq J'^m I.$$
			Hence, for all $m\geq0$ we have
			\begin{align*}
				K'^{m+1} \cap I &= (I+J')^{m+1} \cap I\\
				&= (IK'^m + J'^{m+1}) \cap I\\
				&= IK'^m + J'^{m+1} \cap I\\
				&= IK'^m.
			\end{align*}
			The second assertion can be proved similarly.
		\end{proof}
		We continue the proof of our theorem. Follows from Claim 2 we have
		\begin{align*}
			K'^{m+1}\mm^{N_i-t}+K'^{m+1} \cap I &= K'^{m+1}\mm^{N_i-t} + K'^mI\\
			&= K'^m(K'\mm^{N_i-t} + I)\\
			&= K'^m(J'\mm^{N_i-t}+ I)\\
			&= K'^m(\mm^{N_i}+I)\\
			&= K'^m(\mm^{N_i}+I_N)\\
			&= K'^m(J'\mm^{N_i-t} + I_N)\\
			&= K'^m(K'\mm^{N_i-t} + I_N)\\
			&= K'^{m+1}\mm^{N_i-t} + K'^{m+1} \cap I_N
		\end{align*}
		for all $0 \le i \le t-1$ and all $m \ge 0$. Therefore
		\begin{eqnarray*}
			\frac{K'^{m+1}+I}{K'^{m+1}\mm^{N_i-t}+I} & \cong& \frac{K'^{m+1}}{K'^{m+1}\mm^{N_i-t}+K'^{m+1}\cap I}\\
			&\cong & \frac{K'^{m+1}}{K'^{m+1}\mm^{N_i-t}+K'^{m+1}\cap I_N}\\
			&\cong & \frac{K'^{m+1} +I_N}{K'^{m+1}\mm^{N_i-t}+I_N}
		\end{eqnarray*}
		for all $0 \le i \le t-1$ and all $m \ge 0$. The equality (6) is now clear. The proof is complete.
	\end{proof}
	We close the paper with the following.
	\begin{Remark} If $R$ is Cohen-Macaulay, our formula $N = s!\, e(R/I) + 1$ slightly improves the formula of Srinivas and Trivedi. If $R$ is generalized Cohen-Macaulay but not Cohen-Macaulay, according to the proof we can choose 
		$$N = N_{t-1} = s!\, \hdeg(R/I) + (s+1)I(R)-s$$
	\end{Remark}

\end{document}